\documentclass[10pt,fleqn]{article}

\usepackage{mathptmx}
\usepackage[T1]{fontenc}

\usepackage{latexsym,epsfig,verbatim}
\usepackage{amsmath,amsthm,amssymb}
\usepackage{bm} 

\usepackage{url}
\usepackage{color}
\usepackage[colorlinks,citecolor=blue]{hyperref}

\usepackage[symbol*,stable]{footmisc}
\DefineFNsymbols{richard}[math]{
\circ 				{\circ\circ} 
\bullet 			{\bullet\bullet} 	
\Box				{\Box\!\Box}
\dagger			{\dagger\dagger}
\ddagger			{\ddagger\ddagger}					
}%
\setfnsymbol{richard}
\renewcommand\footnotemark{}

\usepackage{accents}

\usepackage{tikz}
\usetikzlibrary{matrix,arrows,decorations.pathmorphing}

\numberwithin{equation}{section} 

\theoremstyle{plain}
\newtheorem*{theo}{Theorem}
\newtheorem{theorem}{Theorem}
\newtheorem{lemma}[theorem]{Lemma}

\newtheorem{corollary}[theorem]{Corollary}
\newtheorem*{cor}{Corollary}

\newtheorem{proposition}[theorem]{Proposition}
\newtheorem{scholium}[theorem]{Scholium}

\newtheorem*{claim}{Claim}

\theoremstyle{definition}

\newcommand{\propi}{\Pi}

\newcommand{\Mod}{\mathrm{Mod}}
\newcommand{\Pure}{\mathcal{G}}
\newcommand{\Homeo}{\mathrm{Homeo}}

\newcommand{\Aut}{\mathrm{Aut}}
\newcommand{\Out}{\mathrm{Out}}

\newcommand{\checkGamma}{\check{\Pure}}

\newcommand{\QQ}{\mathbb{Q}}

\newcommand{\ZZ}{\mathbb{Z}}

\newcommand{\curv}{\mathcal{C}}
\newcommand{\durv}{\mathcal{D}}
\newcommand{\checkcurv}{\check{\mathcal{C}}}

\newcommand{\calA}{\mathcal{A}}
\newcommand{\calB}{\mathcal{B}}
\newcommand{\calE}{\mathcal{E}}

\newcommand{\calK}{\mathcal{K}}
\newcommand{\calN}{\mathcal{N}}
\newcommand{\calP}{\mathcal{P}}

\newcommand{\calS}{\mathcal{S}}

\newcommand{\calX}{\mathcal{X}}

\newcommand{\rmH}{\mathrm{H}}

\newcommand{\frakD}{\mathfrak{D}}

\newcommand{\dottedS}{\mathring{S}}
\newcommand{\dottedpi}{\mathring{\pi}}
\newcommand{\doubledottedS}{\! \! \! \mathring{\ \, S} \mathring{\! \! \! \! \! \phantom{S}} \, }

\newcommand{\co}{\colon\thinspace}
\newcommand{\twoheadlongrightarrow}{\relbar\joinrel\twoheadrightarrow}

\newcommand{\smallinfinity}{{\hspace{.5pt} \rotatebox{90}{\footnotesize $8$}}}
\newcommand{\scriptinfinity}{{\rotatebox{90}{\scriptsize $8$}}}

\begin{document}

\title{\textbf{Congruence kernels around affine curves}}
\author{Richard Peabody Kent IV\thanks{The author was supported in part by an NSF postdoctoral fellowship and NSF grant DMS-1104871.}}
\date{February 14, 2014}

\maketitle

\begin{abstract}
Let $S$ be a smooth affine algebraic curve, and let $\dottedS$ be the Riemann surface obtained by removing a point from $S$. 
We provide evidence for the congruence subgroup property of mapping class groups by showing that the congruence kernel
\[
\ker\bigg(\widehat{\Mod(\dottedS)} \to \Out\Big(\widehat{\pi_1(\dottedS)}\Big)\bigg)
\]
lies in the centralizer of every braid in $\smash{\Mod(\dottedS)}$.
As  a corollary, we obtain a new proof of Asada's theorem that the congruence subgroup property holds in genus one.
We also obtain simple--connectivity of Boggi's procongruence curve complex $\checkcurv(\dottedS)$ when $S$ is affine, and a new proof of Matsumoto's theorem that the congruence kernel depends only on the genus in the affine case.
\end{abstract}

\section{Introduction}
Let $S$ be a Riemann surface of finite type, and let $\Mod(S) = \pi_0\big(\Homeo^+(S)\big)$ be its mapping class group.
If $C$ is a finite index characteristic subgroup of $\pi = \pi_1(S)$, there is a natural map
\[
\Mod(S) \to \Out(\pi/C),
\]
 and a subgroup containing the kernel of such a map is called a \textit{congruence subgroup} of $\Mod(S)$---the kernels themselves are \textit{principal congruence subgroups}.
Ivanov's \textit{congruence subgroup problem} asks if every finite index subgroup of $\Mod(S)$ is congruence, and we say that $\Mod(S)$ has the \textit{congruence subgroup property} if they are---this is Problem 2.10 of Kirby's List \cite{Kirby}, see \cite{IvanovProblems,Ivanov1987,Ivanov1987.translation}.
The mapping class group is known to possess this property when $S$ has genus no more than two.
In genus zero, the theorem is due to Diaz, Donagi, and Harbater \cite{DiazDonagiHarbater} (see also \cite{McReynoldsThurston} and Section \ref{BirmanInjects} here); in genus one, to Asada \cite{Asada.2001} (see also \cite{BuxErshovRapinchuk}, \cite{BoggiOpen}, and Section \ref{Section:GenusOne} here); and, in genus two, to Boggi \cite{BoggiOpen,BoggiClosed}.

If $G$ is a group, we let $\widehat{G}$ denote its profinite completion.
There is a natural map
\[
\widehat{\, \Mod(S)} \to \Out(\widehat{\pi})
\]
whose kernel $\mathbf{K} = \mathbf{K}(S)$ is the \textit{congruence kernel}. Vanishing of $\mathbf{K}$ is equivalent to the congruence subgroup property.

Due to issues related to torsion, it is advantageous to replace $\Mod(S)$ with a nice subgroup of finite index.
To that end, let $\Pure(S)$ be the congruence subgroup of $\Mod(S)$ consisting of those mapping classes acting trivially on $\mathrm{H}_1(S; \mathbb{F}_3)$.
As intersections of congruence subgroups are congruence (see Lemma \ref{finiteintersections}), a finite index subgroup of $\Mod(S)$ contains a congruence subgroup if and only if its intersection with $\Pure(S)$ does.
So, letting $\widehat{\Pure}(S) = \! \widehat{\, \Pure(S)}$, we need only check injectivity of 
\[
\widehat{\Pure}(S) \to \Out(\widehat{\pi})
\]
to establish the congruence subgroup property.
Thanks to a theorem of Grossman \cite{Grossman}, the group $\Pure(S)$ injects into $\Out(\widehat{\pi})$, and we let $\checkGamma(S)$ be the closure of its image.
By the universal property of profinite completions, $\checkGamma(S)$ is the image of $\widehat{\Pure}(S) \to \Out(\widehat{\pi})$.

Let $\curv(S)$ be the curve complex of $S$---the simplicial flag complex whose vertices are isotopy classes of nonperipheral simple loops on $S$  joined by an edge if they may be realized as disjoint loops on $S$.
In his attack on the congruence subgroup problem \cite{Boggi.2006}, Boggi introduced profinite versions of this complex, subsequently studied in \cite{BoggiLochak,Lochak,Lochak2011}, which we now discuss.

Let $\calA = \calA(S)$ be the inverse system of all finite index subgroups of $\Pure(S)$ under inclusion and let $\calK = \calK(S)$ be the inverse system of all congruence subgroups of $\Pure(S)$.
The group $\Pure(S)$ is torsion--free \cite{Serre}.
Not only that, but $\Pure(S)$ is \textit{pure}, meaning that its elements fix a simplex of $\curv(S)$ if and only if they fix all of the vertices of that simplex, see Corollary 1.8 of \cite{Ivanov}.
It follows that if $\mathfrak{a}$ lies in $\calA$, the quotient
\[
\curv^\mathfrak{a} = \curv(S)/\mathfrak{a}
\]
is naturally a finite simplicial complex.
The \textit{profinite curve complex} is the limit
\[
\widehat{\curv}(S) = \lim_{\calA} \curv^\mathfrak{a}
\]
and the \textit{procongruence curve complex} is 
\[
\checkcurv(S) = \lim_{\calK} \curv^\kappa .
\]
These limits may be taken in the topological category, but the resulting spaces are quite unruly.
A natural solution, that we adopt, is to consider the $\curv^\mathfrak{a}$ simplicial finite sets and take limits in the category of simplicial profinite sets.
The objects of this category are simplicial objects in the category of profinite sets, called \textit{profinite spaces}, and the morphisms are the natural transformations between them. 
See Section \ref{ProfiniteStuff} and \cite{Boggi.2006,Lochak2011,Quick,QuickRemarks}.

Passing to limits, the action of $\Pure(S)$ on $\curv(S)$ yields an action of $\widehat{\Pure}(S)$ on $\widehat{\curv}(S)$ and of $\checkGamma(S)$ on $\checkcurv(S)$.
The zero--skeleton of $\curv(S)$ injects into the zero--skeleta of $\widehat{\curv}(S)$ and $\checkcurv(S)$, and we let context determine of which space a simple loop is to be considered a vertex.

If $G$ is a group acting on a set $X$, we let $G_x$ denote the stabilizer in $G$ of an element $x$ in $X$.
Most often, $X$ is the curve complex or one of its profinite cousins, and $x$ is a vertex.
If $X$ and $Y$ are subsets of $\widehat{\Pure}(S)$ and $\checkGamma(S)$, we let $\overline{X}$ and $Y^\star$ denote their closures, respectively.

Let $\dottedS$ denote the surface obtained from $S$ by removing a point, let $\doubledottedS$ denote the surface obtained from $\dottedS$ by removing a point, and let $\dottedpi = \pi_1(\dottedS)$.

Our main theorem is the following, proven in Section \ref{cornering}.
\begin{theo} Let $S$ be a smooth affine curve.
Let $\calP$ be the set of nonperipheral simple loops on $\dottedS$ that are peripheral in $S$.
Then
\[
\mathbf{K}(\dottedS) \subset \bigcap_{\gamma \, \in \calP} \, \widehat{\Pure(\dottedS)}_\gamma \ .
\]
In particular, the kernel $\smash{\mathbf{K}(\dottedS)}$ lies in the centralizer of every braid in $\smash{\Mod(\dottedS)}$.
\end{theo}
\noindent
In Section \ref{Section:GenusOne}, we show that the congruence subgroup property in genus one follows quickly from this theorem.

The Birman exact sequence \cite{Birman.1969}
\[
\begin{tikzpicture}[>=to, line width = .075em, baseline=(current bounding box.center)]
\matrix (m) [matrix of math nodes, column sep=1.5em, row sep = 1em, text height=1.5ex, text depth=0.25ex]
{
1  & \pi & \Pure(\dottedS) & \Pure(S) &  1 \\
};
\path[->,font=\scriptsize]
(m-1-1) edge 								(m-1-2)
(m-1-2) edge 								(m-1-3)
(m-1-3) edge 								(m-1-4)
(m-1-4) edge 								(m-1-5)
	;
\end{tikzpicture}
\]
has a profinite analog
\[
\begin{tikzpicture}[>=to, line width = .075em, baseline=(current bounding box.center)]
\matrix (m) [matrix of math nodes, column sep=1.5em, row sep = 1em, text height=1.5ex, text depth=0.25ex]
{
1  & \widehat{\pi} & \widehat{\Pure}(\dottedS) & \widehat{\Pure}(S) &  1. \\
};
\path[->,font=\scriptsize]
(m-1-1) edge 								(m-1-2)
(m-1-2) edge 								(m-1-3)
(m-1-3) edge 								(m-1-4)
(m-1-4) edge 								(m-1-5)
	;
\end{tikzpicture}
\]
Our techniques provide a new proof of M. Matsumoto's procongruence version of this sequence (which follows from Theorem 2.2 of \cite{Matsumoto}).
See Section \ref{IndependentKernels}.
\begin{theo}[Matsumoto] For affine curves, the congruence kernel depends only on the genus. 
In other words, $\mathbf{K}(\dottedS) \cong \mathbf{K}(S)$ when $S$ is affine.
In this case, there is a short exact sequence
\[
\begin{tikzpicture}[>= to, line width = .075em, baseline=(current bounding box.center)]
\matrix (m) [matrix of math nodes, column sep=1.5em, row sep = 1em, text height=1.5ex, text depth=0.25ex]
{
1 &  \widehat{\pi}  & \checkGamma(\dottedS) & \checkGamma(S) & 1 .\\
};
\path[->,font=\scriptsize]
(m-1-1) edge (m-1-2)
(m-1-2) edge (m-1-3)
(m-1-3) edge (m-1-4)
(m-1-4) edge (m-1-5)
	;
\end{tikzpicture}
\]
\end{theo}
\noindent It is a more difficult theorem of Boggi \cite{BoggiClosed} and, independently, Hoshi--Mochizuki \cite{HoshiMochizuki2009} that this is still the case when $S$ is projective.
Note that Matsumoto's theorem immediately implies that the congruence subgroup property holds when the genus is zero, as a thrice--punctured sphere has trivial mapping class group.  
In fact, and as first shown by Asada \cite{Asada.2001}, one does not need the full strength of Matsumoto's theorem to obtain the congruence subgroup property in genus zero, see Section \ref{BirmanInjects}.

We conjecture that the braid subgroup $\mathfrak{b}(\dottedS)$ of $\widehat{\Pure}(\dottedS)$ has trivial centralizer.
This may prove difficult to establish, as it is unknown if the center of $\widehat{\Pure}(\dottedS)$ vanishes, though it is known that $\checkGamma(\dottedS)$ is center--free (see Corollary 6.7 of \cite{Lochak2011} for the affine case and Theorem 6.13 of \cite{HoshiMochizuki2012} for the general case).
By the Boggi--Matsumoto short exact sequence
\[
\begin{tikzpicture}[>= to, line width = .075em, baseline=(current bounding box.center)]
\matrix (m) [matrix of math nodes, column sep=1.5em, row sep = 1em, text height=1.5ex, text depth=0.25ex]
{
1 &  \widehat{\pi}  & \checkGamma(\dottedS) & \checkGamma(S) & 1 ,\\
};
\path[->,font=\scriptsize]
(m-1-1) edge (m-1-2)
(m-1-2) edge (m-1-3)
(m-1-3) edge (m-1-4)
(m-1-4) edge (m-1-5)
	;
\end{tikzpicture}
\]
our main theorem would imply the congruence subgroup property in general if, after adding any number of punctures, there were always a braid whose centralizer in $\widehat{\Pure}(\dottedS)$ lied in $\widehat{\pi}$.
We expect that a generic pseudo-Anosov braid in $\pi$ should suffice.
We record these observations in a theorem.
Let $S_{g,n}$ be the closed Riemann surface of genus $g$ with $n$ punctures, and let $\pi_{g,n}$ denote $\pi_1(S_{g,n})$.
\begin{theo}
Suppose that $2g + n \geq 3$, let $k > 1$ be a natural number, and consider the Birman sequence
\[
\begin{tikzpicture}[>=to, line width = .075em, baseline=(current bounding box.center)]
\matrix (m) [matrix of math nodes, column sep=1.5em, row sep = 1em, text height=1.5ex, text depth=0.25ex]
{
1  & \widehat{\pi}_{g,\, n+k-1} & \widehat{\Pure}(S_{g,\, n+k}) & \widehat{\Pure}(S_{g,\, n+k-1}) &  1. \\
};
\path[->,font=\scriptsize]
(m-1-1) edge 								(m-1-2)
(m-1-2) edge 								(m-1-3)
(m-1-3) edge 								(m-1-4)
(m-1-4) edge 								(m-1-5)
	;
\end{tikzpicture}
\]
Then $\Mod(S_{g,n})$ has the congruence subgroup property if there is a $B$ in $\pi_{g,\, n+k-1}$ whose centralizer in $\widehat{\Pure}(S_{g,\, n+k})$ lies in $\widehat{\pi}_{g,\, n+k-1}$. 
For example, a $B$ in $\pi_{g,\, n+k-1}$ whose centralizer equals $\overline{\langle B \rangle}$ implies the congruence subgroup property for $\Mod(S_{g,n})$.
\qed
\end{theo}

\begin{center}
. \quad . \quad .
\end{center}
\noindent
G. Quick has developed the homotopy theory of profinite spaces \cite{Quick}.
In particular, he defines a covariant functor $\Pi_1(\ \cdot \ )$ serving as the fundamental group that assigns a profinite group $\Pi_1(\mathcal{X})$ to a profinite space $\mathcal{X}$.\footnote{Quick uses the notation $\pi_n(\, \cdot \, )$ for his homotopy groups, but, to avoid confusion, we do not.}
For a simplicial finite set $\mathcal{X}$, the group $\Pi_1(\mathcal{X})$ agrees with the profinite completion $\widehat{\pi}_1(\mathcal{X})$ of the usual fundamental group $\pi_1(\mathcal{X})$.
Furthermore, 
the fundamental group $\Pi_1(\, \cdot \, )$ commutes with cofiltered limits of simplicial finite sets.
See Section \ref{ProfiniteStuff}.

Boggi has observed that $\Pi_1(\widehat{\curv}(S)) = \lim_\calA  \widehat{\pi}_1\big(\curv^\mathfrak{a}(S)\big)$ vanishes unless $S$ is of small complexity, see Theorem \ref{BoggiFundamentalGroup}.
He has also shown that the Congruence Subgroup Problem holds for all surfaces $S$ if and only if the profinite space $\checkcurv(F)$ is simply--connected for all surfaces $F$, see Corollary 7.2 of \cite{Boggi1111.2372}.\footnote{It is important to emphasize that simple--connectivity of $\checkcurv(S)$ alone is not enough to deduce the congruence subgroup property for $\Mod(S)$: one really needs simple--connectivity of $\checkcurv(F)$ for all surfaces $F$ of lower or equal complexity.}
We refer the reader to \cite{Lochak2011} for an in--depth discussion of the interaction of profinite Teichm\"uller theory and Quick's homotopy theory.

The proof of our main theorem shows that $\checkcurv(\dottedS)$ is simply--connected when $S$ is affine, see Theorem \ref{sequester}.
\begin{theo}  Let $S=S_{g,n}$ be a smooth affine curve with $\chi(S) < 0$ and $3g-3 + n \geq 1$.
Then $\checkcurv(\mathring{S})$ is simply--connected:
\[
\Pi_1(\checkcurv(\mathring{S})) = \lim_{\calK} \, \widehat{\pi}_1(\curv^\kappa(\dottedS)) 
= 1.
\]
\end{theo}
 
\begin{cor}
If $S=S_{g,n}$ is affine with $\chi(S) < 0$ and $3g-3 + n \geq 1$, then, for any prime number $p$,
\[
\rmH_1(\checkcurv(\dottedS) ; \mathbb{F}_p) 
:= \lim_\calK \, \rmH_1 \big(\curv^\kappa(\dottedS) ; \mathbb{F}_p \big)
= 0
.
\]
\end{cor}

In Section \ref{Section:Thoughts}, we record some thoughts on the fundamental groups of the finite--level complexes $\curv^\mathfrak{a}$.

\medskip
\noindent \textbf{Acknowledgments.} The author thanks  
Marco Boggi, 
Tom Church,
Jordan Ellenberg,
Chris Leininger, 
Pierre Lochak,
Ben McReynolds, 
Andy Putman, 
Gereon Quick,
and 
Ben Wieland 
for many useful conversations.
The author thanks the referees for their careful readings and many helpful comments.  The author also thanks the Institute for Advanced Study and the Park City Math Institute, where some of this work was carried out.

\section{Point pushing}

We make extensive use of the Birman exact sequence \cite{Birman.1969}:
\begin{equation}\label{BirmanSequence.Equation}
\begin{tikzpicture}[>=to, line width = .075em, baseline=(current bounding box.center)]
\matrix (m) [matrix of math nodes, column sep=1.5em, row sep = 1em, text height=1.5ex, text depth=0.25ex]
{
1  & \pi & \Pure(\dottedS) & \Pure(S) &  1. \\
};
\path[->,font=\scriptsize]
(m-1-1) edge 								(m-1-2)
(m-1-2) edge 		node[auto] {$\tau$}		(m-1-3)
(m-1-3) edge 		node[auto] {$F$}			(m-1-4)
(m-1-4) edge 								(m-1-5)
	;
\end{tikzpicture}
\end{equation}
The map $F$ is the natural map obtained by ``forgetting the puncture,'' or, in other words, by extending homeomorphisms from $\dottedS$ to $S$.
The image of $\pi$ in $\Pure(\dottedS)$ is the subgroup  $\calB$ of elements represented by  \textit{point--pushing} homeomorphisms, defined as follows---see Section 4.2 of \cite{FarbMargalit} for details.
Let $\{x\} = S - \dottedS$, and pick an isomorphism $\pi_1(S) \cong \pi_1(S,x)$.
A loop $h \co [0,1] \to S$ based at $x$ is an isotopy of the inclusion map $\{x\} \to S$.
Such an isotopy may be extended to an ambient isotopy $H \co S \times [0,1] \to S$.
The homeomorphism $H( \, \cdot \, , 1)$ at time one is a homeomorphism of $S$ fixing $x$, which induces a homeomorphism $\mathring{H}$ of $\dottedS$.
We call $H(\, \cdot \, , 1)$ and $\mathring{H}$ \textit{point--pushing homeomorphisms}.
The isotopy class of $\mathring{H}$ depends only on the pointed homotopy class of $h$, and so yields a well defined element of $\Mod(\dottedS)$.
This mapping class acts trivially on homology, and so lies in $\Pure(\dottedS)$.

The Birman exact sequence fits naturally into a commutative diagram 
\[
\begin{tikzpicture}[>=to, line width = .075em, baseline=(current bounding box.center)]
\matrix (m) [matrix of math nodes, row sep=1em,
column sep=2em, text height=1.5ex, text depth=0.25ex]
{
1   & \pi & \Pure(\dottedS) & \Pure(S) &  1 \\
1   & \mathrm{Inn}(\pi)  & \Aut(\pi)  & \Out(\pi) & 1 \\
};
\path[->,font=\scriptsize]
(m-1-1)	edge						(m-1-2)
(m-1-2)	edge						(m-1-3)
		edge						(m-2-2)
(m-1-3)	edge	node[auto] {$F$}		(m-1-4)
		edge						(m-2-3)
(m-1-4)	edge						(m-1-5)
		edge						(m-2-4)
(m-2-1)	edge						(m-2-2)
(m-2-2)	edge						(m-2-3)
(m-2-3)	edge						(m-2-4)
(m-2-4)	edge						(m-2-5)
	;
\end{tikzpicture}
\]
where the vertical maps are injections and the map $\pi \to \mathrm{Inn}(\pi)$ is the natural isomorphism.
A consequence is that every inner automorphism of $\pi_1(S,x)$ is realized by a point--pushing homeomorphism fixing $x$.

\section{Pulling and pushing congruence subgroups}

\begin{lemma}\label{finiteintersections} Finite intersections of congruence subgroups are congruence.
\end{lemma}
\begin{proof} It suffices to prove that the intersection of two principal congruence subgroups is congruence.

Let $C$ and $D$ be finite index characteristic subgroups of $\pi$, and let 
\[
E = \bigcap_{\varphi \in \Aut(\pi)} \varphi(C \cap D)
\]
be the characteristic core of $C \cap D$.
Our maps to $\Out(\pi/C)$ and $\Out(\pi/D)$ both factor through $\Out(\pi/E)$:
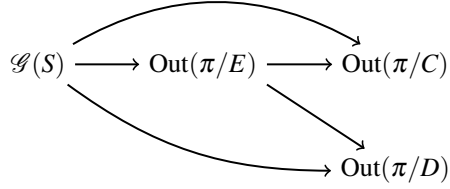
\[
\begin{tikzpicture}[>=to, line width = .075em, baseline=(current bounding box.center)]
\matrix (m) [matrix of math nodes, row sep=2.5em,
column sep=2.5em, text height=1.5ex, text depth=0.25ex]
{ \Pure(S) & \Out(\pi/E) & \Out(\pi/C)   \\
 & & \Out(\pi/D) \\ };
\path[->]
(m-1-1)				edge 			(m-1-2)
					edge [bend left=30] 	(m-1-3)
					edge [bend right=18] 	(m-2-3)
(m-1-2)				edge 			(m-1-3)
(m-1-2.south east)		edge				(m-2-3);
\end{tikzpicture}
\]
and so $\ker\big(\Pure(S) \to \Out(\pi/E)\big)$ lies in the intersection of $\ker\big(\Pure(S) \to \Out(\pi/C)\big)$ and $\ker\big(\Pure(S) \to \Out(\pi/D)\big)$.
\end{proof}

Lemma \ref{finiteintersections} shows that $\calK$ forms an inverse system and allows us to replace $\Mod(S)$ with $\Pure(S)$ when considering the congruence subgroup problem, replacing congruence subgroups of $\Mod(S)$ with their intersections with $\Pure(S)$, called \textbf{congruence subgroups} of $\Pure(S)$.

\begin{lemma}\label{pullcongruence} If $\kappa$ is a congruence subgroup of $\Pure(S)$, then $F^{-1}(\kappa)$ is a congruence subgroup of $\Pure(\dottedS)$.
\end{lemma}
\noindent 
A loop on a Riemann surface of finite type is \textit{peripheral} if it is freely homotopic into a subsurface that is conformally equivalent to a punctured disk.
An element of the fundamental group is \textit{peripheral} if its representatives are.
Note that we consider a null--homotopic loop to be peripheral.
Letting $\iota \co \dottedS \to S$ be the inclusion map, we have the following short exact sequence
\begin{equation}
\begin{tikzpicture}[>=to, line width = .075em, baseline=(current bounding box.center)]
\matrix (m) [matrix of math nodes, column sep=1.5em, row sep = 1em, text height=1.5ex, text depth=0.25ex]
{
1  & \calN & \dottedpi & \pi &  1 \\
};
\path[->, font=\scriptsize]
(m-1-1) edge 			(m-1-2)
(m-1-2) edge 			(m-1-3)
(m-1-3) edge node[auto] {$\iota_*$}	(m-1-4)
(m-1-4) edge 			(m-1-5)
	;
\end{tikzpicture}
\end{equation}
where $\calN$ is the normal closure of a peripheral element corresponding to the distinguished puncture of $\dottedS$ (the puncture located at $\{x\}=S - \dottedS$).
\begin{proof}[Proof of Lemma \ref{pullcongruence}]
Pick a point $z$ on $\dottedS$.
Identify $\dottedpi$ with $\pi_1(\dottedS, z)$ and $\pi$ with $\pi_1(S, z)$.

Let $D$ be a characteristic subgroup of $\pi$ and let $\kappa$ be the associated principal congruence subgroup.

Let $\varphi$ lie in $\Pure(\dottedS)$ and let $H \co S \to S$ be a homeomorphism fixing $x$ and $z$ such that $h = H\big|_{\dottedS}$ represents $\varphi$.
Let $c$ lie in $\iota_*^{-1}(D)$.
We have $\iota_*h_*(c) = H_*\iota_*(c)$.
Since $D$ is characteristic, $H_*\iota_*(c)$ lies in $D$, and so $h_*(c)$ lies in $\iota_*^{-1}(D)$.
Therefore $\iota_*^{-1}(D)$ is $h_*$--invariant, and so the conjugacy class of $\iota_*^{-1}(D)$ is $\Pure(\dottedS)$--invariant.
We thus have a representation
\[
\rho\co \Pure(\dottedS) \to \Out\big(\dottedpi/\iota_*^{-1}(D)\big)
\]
that fits into a commuting diagram:
\[
\begin{tikzpicture}[>=to, line width = .075em, baseline=(current bounding box.center), descr/.style={fill=white,inner sep=2.5pt}]
\matrix (m) [matrix of math nodes, row sep=2.5em,
column sep=2.5em, text height=1.5ex, text depth=0.25ex]
{ & F^{-1}(\kappa) & \Pure(\dottedS)& \Out\big(\dottedpi/\iota_*^{-1}(D)\big)\\
1  & \kappa & \Pure(S) & \Out\big(\pi/D\big) \\
};
\path[->,font=\scriptsize]
(m-1-2)	edge						(m-1-3)
		edge						(m-2-2)
(m-1-3)	edge node[auto] {$\rho$}		(m-1-4)
		edge node[auto] {$F$}		(m-2-3)
(m-1-4)	edge						(m-2-4)
(m-2-1)	edge						(m-2-2)
(m-2-2)	edge						(m-2-3)
(m-2-3)	edge						(m-2-4)		
		;
\end{tikzpicture}
\]
Exactness of the bottom row yields $F^{-1}(\kappa) \supset \ker \rho$.
So $F^{-1}(\kappa)$ is congruence.\footnote{While the conjugacy class of $\iota_*^{-1}(D)$ is $\Pure(\dottedS)$--invariant, it is not necessarily characteristic.  
To see that this is not an issue, consider the characteristic core $\frakD$ of $\iota_*^{-1}(D)$.
Our kernel $\ker \rho  \subset F^{-1}(\kappa)$ contains the kernel of the representation $\Pure(\dottedS) \to \Out\big(\dottedpi/\frakD)$, which is a principal congruence subgroup in the usual sense.}
\end{proof}

If $M$ and $N$ are subgroups of a group $G$, we let
\[
MN = M \cdot N =  \{ mn \ | \ m \in M \mbox{\ and\ } n \in N \}.
\]
If $G$ acts on a set $X$, our notation for stabilizers becomes ambiguous when applied to products $MN \subset G$.
To remedy this, we adopt the convention that, whenever we write $MN_x$, we always mean $MN_x = M\cdot (N_x)$, and not $(MN)_x$.

\begin{theorem}\label{product} 
Let $S$ be a smooth affine curve, and let $\gamma$ be a nonperipheral simple loop in $\dottedS$ that is peripheral in $S$.
Let $\kappa$ be a congruence subgroup of $\Pure(\dottedS)$.
Let $\calB^\kappa= \kappa \cap \calB$ and $\kappa_\gamma = \kappa \cap \Pure(\dottedS)_\gamma$.
Then $\calB^\kappa \kappa_\gamma = \calB^\kappa \cdot (\kappa_\gamma)$ is a congruence subgroup.
\end{theorem}

As $\calB^\kappa \kappa_\gamma \subset \kappa$, the following corollary is immediate.
\begin{corollary}\label{CofinalCorollary} For $S$ a smooth affine curve, the inverse system of subgroups
\[
\{\calB^\kappa \kappa_\gamma \ \big| \ \kappa \ \mathrm{in}\  \calK(\dottedS) \} 
\]
is cofinal in $\calK(\dottedS)$. \qed
\end{corollary}

\begin{lemma}\label{BKlemma} 
Let $S$ be a smooth affine curve, and let $\gamma$ be a nonperipheral simple loop in $\dottedS$ that is peripheral in $S$.
Let $\kappa$ be a congruence subgroup of $\Pure(\dottedS)$.
Then $\calB \kappa_\gamma$ is a congruence subgroup of $\Pure(\dottedS)$.
\end{lemma}

\begin{proof}
As $\calB$ is normal in $\Pure(\dottedS)$, the set $\calB\mathfrak{a}_\gamma$ is a subgroup for any $\mathfrak{a}$ in $\calA$. 

Let $\kappa$ be a congruence subgroup containing a principal congruence subgroup 
\[
\mathfrak{p} = \ker\left( \Pure(\dottedS) \to \Out(\dottedpi/C)\right)
\]
for some characteristic subgroup $C$ of $\dottedpi$.

Let $\iota \co \dottedS \to S$ be the inclusion map.
By our choice of $\gamma$, there is a $\pi_1$--injective subsurface $\Sigma \subset \dottedS$ with $\partial \Sigma = \gamma$ such that 
\[
j = \iota_*\big|_{\pi_1(\Sigma)} \co \pi_1(\Sigma) \longrightarrow \pi
\]
is an isomorphism.
Here we choose a  basepoint $z$ in $\gamma$ and identify $\dottedpi$ and $\pi$ with $\pi_1(\dottedS, z)$ and $\pi_1(S, z)$ respectively.

We lift $\Pure(\dottedS)_\gamma$ to $\Aut(\dottedpi)$ by identifying the former with the group of homeomorphisms of $\dottedS$ that are the identity on $\overline{\dottedS - \Sigma}$ up to isotopy relative to $\overline{\dottedS - \Sigma}$.
Since $C$ is characteristic and $\pi_1(\Sigma,z)$ is $\Pure(\dottedS)_\gamma$--invariant, the subgroup $D = C \cap \pi_1(\Sigma, z)$ is $\Pure(\dottedS)_\gamma$--invariant.
It follows that the conjugacy class of $\iota_*(D)$ is $\Pure(S)$--invariant, and we have an exact sequence 
\[
\begin{tikzpicture}[>=to, line width = .075em, baseline=(current bounding box.center)]
\matrix (m) [matrix of math nodes, column sep=1.5em, row sep = 1em, text height=1.5ex, text depth=0.25ex]
{
1  & \mathfrak{D} & \Pure(S) & \Out \big(\pi/\iota_*(D)\big). \\
};
\path[->,font=\scriptsize]
(m-1-1) edge 			(m-1-2)
(m-1-2) edge 			(m-1-3)
(m-1-3) edge 			(m-1-4)
	;
\end{tikzpicture}
\]
The subgroup $\mathfrak{D}$ is a $\Pure(S)$--congruence subgroup, and we have the following claim.
\begin{claim}
$
\mathfrak{D} \subset F(\mathfrak{p}_\gamma).
$
\end{claim}
\begin{proof}[Proof of claim]
Let $\psi$ be an element of $\mathfrak{D}$.
Choose a homeomorphism $h \co S \to S$ that fixes $z$ and represents $\psi$.
As mentioned above, the conjugacy class of $\iota_*(D)$ is $\Pure(S)$--invariant.
By postcomposing with a point--pushing homeomorphism fixing $z$, we may assume that $h$ represents an automorphism  $h_*$ of $\pi_1(S,z)$ that preserves $\iota_*(D)$.
Since $\psi$ is an element of $\mathfrak{D}$, this automorphism $h_*$ descends to an inner automorphism of $\pi_1(S,z)/\iota_*(D)$.
By further postcomposing with a point--pushing homeomorphism fixing $z$, we may, and do, assume that $h_*$ in fact represents the trivial automorphism of $\pi_1(S,z)/\iota_*(D)$.
We further assume that $h$ is the identity on $S - \iota(\Sigma)$, which may be achieved by replacing $h$ with a map isotopic to $h$ relative to $z$.

Restricting $h$ to $\iota(\dottedS)$ yields a homeomorphism $H \co \dottedS \to \dottedS$ that fixes $\gamma$ pointwise, and hence represents an element $\varphi$ of $\Pure(\dottedS)_\gamma$ that maps to $\psi$ under the forgetful map $F \co \Pure(\dottedS) \to \Pure(S)$.

Now, 
\[
\pi_1(\dottedS, z) \cong A *_\ZZ B
\]
where $A$ is the free group $\smash{\pi_1\big(\overline{\dottedS - \Sigma}, z\big)}$, $B = \pi_1(\Sigma,z)$, and the amalgamating $\ZZ$ is $\pi_1(\gamma, z)$.
Each element of $\pi_1(\dottedS, z)$ may then be written in a normal form
\[
a_1 b_1 \cdots a_n b_n
\]
where each $a_i$ lies in $A$ and each $b_i$ lies in $B$.
The homeomorphism $H$ induces an automorphism $H_*$ of $\pi_1(\dottedS, z)$, and since $H$ is the identity on $\overline{\dottedS - \Sigma}$, we have
\[
H_*(a_1 b_1 \cdots a_n b_n) = a_1 H_*(b_1) \cdots a_n H_*(b_n).
\]
Now, by construction, $j H_* \big|_B j^{-1}= h_*$, and so $H_* \big|_B$ is trivial in $\Aut(B/ D)$.
So, for each $b$ in $B$, there is a $d$ in $D$ such that $H_*(b) = bd$.
So, for each $i$ there is a $d_i$ in $D$ such that 
\begin{align*}
H_*(a_1 b_1 \cdots a_n b_n) 
& = a_1 H_*(b_1) \cdots a_n H_*(b_n)\\
& = a_1 (b_1 d_1) \cdots a_n (b_n d_n).
\end{align*}
Since $D \subset C$, we conclude that the automorphism $H_*$ induces the trivial automorphism of $\pi_1(\dottedS, z) / C$.
But this means that $\varphi$ lies in $\mathfrak{p}$, and as it also lies in $\Pure(\dottedS)_\gamma$, it lies in $\mathfrak{p}_\gamma$.

Since $F(\varphi) = \psi$, we conclude that $\psi$ is in $F(\mathfrak{p}_\gamma)$.
Since $\psi$ was arbitrary, we conclude that $\mathfrak{D} \subset F(\mathfrak{p}_\gamma)$.
\end{proof}

Since $\mathfrak{D}$ is congruence, and 
\[
\calB \kappa_\gamma 
\supset \calB \mathfrak{p}_\gamma
= F^{-1}(F(\mathfrak{p}_\gamma)) 
\supset F^{-1}(\mathfrak{D}),
\]
the subgroup $\calB \kappa_\gamma$ is congruence by Lemma \ref{pullcongruence}.
\end{proof}

The claim in the proof of Lemma \ref{BKlemma} gives us the following.
\begin{scholium}\label{pushcongruence} If $\kappa$ is a congruence subgroup of $\Pure(\dottedS)$, then $F(\kappa_\gamma)$ is a congruence subgroup of $\Pure(S)$. \qed
\end{scholium}

\begin{proof}[Proof of Theorem \ref{product}]
The set $\calB^\kappa \kappa_\gamma$ lies in $\kappa$, and hence in $\kappa \cap \calB \kappa_\gamma$.
Note that, as $\calB^\kappa$ is not necessarily normal in $\Pure(\dottedS)$, it is not clear that $\calB^\kappa \kappa_\gamma$ is a subgroup of $\Pure(\dottedS)$---we are not requiring $\kappa$ to be normal.
We claim that $\calB^\kappa \kappa_\gamma = \kappa \cap \calB \kappa_\gamma$, which will establish that $\calB^\kappa \kappa_\gamma$ is in fact a subgroup, and, by Lemma \ref{BKlemma}, a congruence one.
To see this, let $b$ in $\calB$ and $k$ in $\kappa_\gamma$ be such that $bk$ lies in $\kappa$.
Since $\kappa_\gamma$ lies in $\kappa$, the element $b$ must lie in $\kappa$.
So $b$ lies in $\calB^\kappa = \kappa \cap \calB$.
\end{proof}

\section{The congruence topology on the Birman kernel}\label{BirmanInjects}
The group $\Pure(\dottedS)$ not only embeds in $\Out(\dottedpi)$, but also as a subgroup of $\Aut(\pi)$.
The first embedding gives rise to the geometric completion $\checkGamma(\dottedS)$, by taking the closure of $\Pure(\dottedS)$ in $\Out(\widehat{\dottedpi})$.
The second embedding gives us a completion $\widetilde \Pure(\dottedS)$, by taking the closure in $\Aut(\widehat{\pi})$. 

By Proposition 3 of \cite{Anderson.1974}, if $G$ is any finitely generated group, the kernel of $\widehat{G} \to \Aut(\widehat{G})$ lies in the center of $\widehat{G}$.
Now, the center of $\widehat{\pi}$ is trivial---thanks to a theorem of Anderson, Proposition 18 of \cite{Anderson.1974}; and, independently, Nakamura,  Corollary 1.3.4 of \cite{Nakamura}.
So there is a natural short exact sequence 
\[
\begin{tikzpicture}[>=to, line width = .075em, baseline=(current bounding box.center)]
\matrix (m) [matrix of math nodes, column sep=1.5em, row sep = 1em, text height=1.5ex, text depth=0.25ex]
{
1 & \widehat{\pi} & \Aut(\widehat{\pi}) & \Out(\widehat{\pi}) & 1 .\\
};
\path[->,font=\scriptsize]
(m-1-1) edge 										(m-1-2)
(m-1-2) edge 		node[auto] {$\widetilde{\tau}$}		(m-1-3)
(m-1-3) edge 										(m-1-4)
(m-1-4) edge										(m-1-5)
	;
\end{tikzpicture}
\]
This gives us a short exact sequence
\begin{equation}\label{TwiddleSequence}
\begin{tikzpicture}[>=to, line width = .075em, baseline=(current bounding box.center)]
\matrix (m) [matrix of math nodes, column sep=1.5em, row sep = 1em, text height=1.5ex, text depth=0.25ex]
{
1 & \widehat{\pi} & \widetilde{\Pure}(\dottedS) & \checkGamma(S) & 1 . \\
};
\path[->,font=\scriptsize]
(m-1-1) edge 									(m-1-2)
(m-1-2) edge 		node[auto] {$\widetilde{\tau}$}	(m-1-3)
(m-1-3) edge 									(m-1-4)
(m-1-4) edge									(m-1-5)
	;
\end{tikzpicture}
\end{equation}
When $S$ is affine, it is a theorem of Matsumoto that $\smash{\widetilde \Pure(\dottedS) \cong \checkGamma(\dottedS)}$, and hence that there is a short exact sequence
\[
\begin{tikzpicture}[>=to, line width = .075em, baseline=(current bounding box.center)]
\matrix (m) [matrix of math nodes, column sep=1.5em, row sep = 1em, text height=1.5ex, text depth=0.25ex]
{
1 & \widehat{\pi} & \checkGamma(\dottedS) & \checkGamma(S) & 1 . \\
};
\path[->,font=\scriptsize]
(m-1-1) edge 			(m-1-2)
(m-1-2) edge 			(m-1-3)
(m-1-3) edge 			(m-1-4)
(m-1-4) edge			(m-1-5)
	;
\end{tikzpicture}
\]
This follows from Theorem 2.2 of \cite{Matsumoto}, and we will provide a new proof of this in Section \ref{IndependentKernels}.
We need the fact that $\widehat{\pi}$ injects into $\checkGamma(\dottedS)$---in other words, that the closure $\calB^\star$ of $\calB$ in $\checkGamma(\dottedS)$ is isomorphic to $\widehat{\calB}$. 
For this, one constructs a natural epimorphism $\checkGamma(\dottedS) \to \widetilde \Pure(\dottedS)$.

\begin{theorem}[Matsumoto]\label{birmanseparable} The closure $\calB^\star$ of $\calB$ in $\checkGamma(\dottedS)$ is isomorphic to $\widehat{\calB}$. 
\end{theorem}
\begin{proof}[Proof (Asada)]
Let $\tau \co \pi \to \Pure(\dottedS)$ be as in the Birman exact sequence \eqref{BirmanSequence.Equation}, and let $\check{\tau} \co \widehat{\pi} \to \checkGamma(\dottedS)$ be its natural extension.
The image of $\check{\tau}$ is $\calB^\star$.
We claim that $\widetilde{\tau} \co \widehat{\pi} \to \widetilde{\Pure}(\dottedS)$ factors through $\check{\tau} \co \widehat{\pi} \to \checkGamma(\dottedS)$.
This will imply that $\check{\tau}$ is injective, and hence that $\calB^\star \cong \widehat{\pi}$.
So we find a commutative triangle
\[
\begin{tikzpicture}[>=to, line width = .075em, baseline=(current bounding box.center)]
\matrix (m) [matrix of math nodes, column sep=3em, row sep = 2em, text height=1.5ex, text depth=0.25ex]
{
& \checkGamma(\dottedS) \\
\widehat{\pi} & \widetilde{\Pure}(\dottedS) \\
};
\path[->,font=\scriptsize]
(m-1-2) edge 		node[auto] {$\xi$}	(m-2-2)
(m-2-1) edge 		node[below] {$\widetilde{\tau}$}	(m-2-2)
(m-2-1) edge		node[auto] {$\check{\tau}$}	(m-1-2)
	;
\end{tikzpicture}
\]
The argument here is borrowed from the proof of Theorem 1 of \cite{Asada.2001}.

Let $\widetilde{F} \co \widetilde{\Pure}(\doubledottedS) \to \widetilde{\Pure}(\dottedS)$ be the natural projection.
Given an element $\varphi$ of $\checkGamma(\dottedS) \subset \Out\big(\widehat{\dottedpi}\big)$, we may lift it to an element $\widetilde \varphi$ of $\widetilde{\Pure}(\doubledottedS) \subset \Aut\big(\widehat{\dottedpi}\big)$ that fixes the element $c$ corresponding to the new puncture of $\doubledottedS$, and then project to an element $\widetilde{F}(\widetilde{\varphi})$ of $\Aut(\widehat{\pi})$.
Any two lifts $\widetilde{\varphi}$ and $\widetilde{\varphi}'$ of $\varphi$ differ by an inner automorphism centralizing $c$.
This centralizer is topologically generated by $c$, by \cite{HerfortRibes} (see also Lemma 2.1.2 of \cite{Nakamura}), and so we have a well--defined map $\xi \co \checkGamma(\dottedS) \to \widetilde{\Pure}(\dottedS)$ given by $\xi(\varphi) = \widetilde{F}(\widetilde{\varphi})$. 

The map $\xi$ is clearly a homomorphism. 
To see that it is continuous, take a sequence $\varphi_n$ in $\checkGamma(\dottedS)$ converging to some $\varphi$. 
Pick lifts $\widetilde{\varphi}_n$ of the $\varphi_n$ to $\Aut(\widehat{\dottedpi})$ fixing $c$.
After passing to a subsequence, the $\widetilde{\varphi}_n$ converge to some $\widetilde{\varphi}_\smallinfinity$.  
Since the projection $\widetilde{\Pure}(\doubledottedS) \to \checkGamma(\dottedS)$ is continuous, the element $\widetilde{\varphi}_\smallinfinity$ is a lift of $\varphi$.
By continuity of multiplication, this $\widetilde{\varphi}_\smallinfinity$  fixes $c$, and since $\widetilde{F}$ is continuous, we conclude that $\xi$ is continuous.\footnote{Alternatively, one may invoke the deep theorem of Nikolov and Segal \cite{Nikolovsegal} that homomorphisms from topologically finitely generated profinite groups to arbitrary profinite groups are always continuous.}

Now, every inner automorphism of $\pi$ is realized by a point--pushing homeomorphism of $S$ which is the identity in a neighborhood of the basepoint.
Deleting the fixed points of such homeomorphisms and passing to isotopy classes produces the subgroup $\tau(\pi) = \calB  \subset \Pure(\dottedS)$, and so 
\[
\xi \circ \check{\tau}\big|_\pi = \widetilde{\tau}\big|_\pi.
\]
By the universal property of profinite completions, we have $\xi \circ \check{\tau}= \widetilde{\tau}$.

Since $\xi$ is continuous, the closure $\calB^\star$ of $\calB = \tau(\pi)$ in $\checkGamma(\dottedS)$ surjects the closed subgroup $\widetilde{\tau}(\widehat{\pi})$ of $\Aut(\widehat{\pi})$, as $\widetilde{\tau}(\pi)$ is dense in the latter.
As $\calB^\star = \check{\tau}(\widehat{\pi})$ and $\widetilde{\tau}$ is injective, we have $\calB^\star \cong \widehat{\pi}$.
\end{proof}

See \cite{McReynoldsThurston} for a different proof.

The continuous epimorphism $\checkGamma(\dottedS) \to \widetilde \Pure(\dottedS)$, the short exact sequences 
\[
\begin{tikzpicture}[>=to, line width = .075em, baseline=(current bounding box.center)]
\matrix (m) [matrix of math nodes, column sep=1.5em, row sep = 1em, text height=1.5ex, text depth=0.25ex]
{
1 & \widehat{\pi} & \widehat{\Pure}(\dottedS) & \widehat{\Pure}(S) & 1,  \\
};
\path[->,font=\scriptsize]
(m-1-1) edge 			(m-1-2)
(m-1-2) edge 			(m-1-3)
(m-1-3) edge 			(m-1-4)
(m-1-4) edge			(m-1-5)
	;
\end{tikzpicture}
\]
and \eqref{TwiddleSequence} are all that is needed to establish the congruence subgroup property for mapping class groups of punctured spheres.
See the proof of Theorem 1 in \cite{Asada.2001}, Section 4 of \cite{Boggi.2006}, or Section 6 of \cite{Lochak2011}.

\section{Profinite spaces}\label{ProfiniteStuff}

We quickly review some of the notions from \cite{Quick}. 
See also \cite{Boggi.2006,Lochak2011,QuickRemarks}.

A \textit{profinite set} is an inverse limit (in the topological category) of discrete finite sets.
Profinite sets form a category $\widehat{\calE}$ with continuous maps as morphisms.
A \textit{profinite space} is a simplicial object in this category---a contravariant functor from the simplex category $\Delta$ to $\widehat{\calE}$.
Profinite spaces form a category $\widehat{\calS}$ whose morphisms are the natural transformations.
We let $\calS$ denote the category of simplicial sets.
There is a forgetful functor $| \, \cdot \, | \co \widehat{\calS} \to \calS$ that sends a profinite space to its underlying simplicial set.

A profinite space $\calX$ may be considered a sequence $\calX_\bullet$ of profinite sets $\{\calX_n\}_{n=0}^\scriptinfinity$, called the \textit{skeleta}, together with all compositions of face $d_i \co \calX_n \to \calX_{n-1}$ and degeneracy maps $s_i \co \calX_n \to \calX_{n+1}$.
A group $G$ \textit{acts} on $\calX$ if it acts on the $\calX_n$ equivariantly respecting  face and degeneracy maps.
If $G$ is a topological group, we say that a $G$--action is \textit{continuous} if $G$ acts continuously on the $\calX_n$.
Since singletons are closed in profinite sets, stabilizers $G_x$ are closed for continuous actions.

The action of $\Pure(S)$ on the simplicial set $\curv(S)$ extends naturally to continuous actions of $\widehat{\Pure}(S)$ on $\widehat{\curv}(S)$ and $\checkGamma(S)$ on $\checkcurv(S)$.
The forgetful functor $| \, \cdot \, |$ has a left adjoint $(\, \cdot \, )^\circ \co \calS \to \widehat{\calS}$ called \textit{profinite completion}.\footnote{Quick \cite{Quick} uses the notation $\widehat{(\, \cdot \, )}$ for this functor, which conflicts with our use of that notation.}

We warn the reader that the profinite completion $\curv(S)^\circ$ of the simplicial set $\curv(S)$ is not the same as Boggi's profinite curve complex $\widehat{\curv}(S)$ of \cite{Boggi.2006}.
On the other hand, there is a variant of $(\, \cdot \, )^\circ$ adapted to the equivariant setting of discrete groups acting on simplicial sets that does produce $\widehat{\curv}(S)$: see section 4 of \cite{QuickRemarks}, particularly remark 4.5 there.
As we will not need to deal with the subtleties relating the two definitions, we refer the reader to \cite{Quick} and \cite{QuickRemarks} for details.

In \cite{Quick}, the homotopy theory of profinite spaces is developed by giving $\widehat{\calS}$ the structure of a model category, which allows us to discuss the homotopy type of the profinite spaces $\widehat{\curv}(S)$ and $\checkcurv(S)$. 
Quick defines profinite homotopy groups $\propi_n(\, \cdot \, )$ on the category $\widehat{\calS}$, and we say that a profinite space $\calX$ is \textit{simply--connected} if $\propi_0(\calX)$ and $\propi_1(\calX)$ vanish.

If $\calX$ is a simplicial finite set, then $\Pi_1(\calX)$ is the profinite completion $\widehat{\pi}_1(\calX)$ of the usual fundamental group $\pi_1(\calX)$ of $\calX$, by Proposition 2.1 of \cite{Quick}.
Furthermore, it follows from the proof of Theorem 3.11 of \cite{QuickRemarks} that $\Pi_1(\, \cdot \, )$ commutes with cofiltered limits of simplicial finite sets, see Proposition 7.1 of \cite{Boggi1111.2372}.
This means that if 
\[
\{\calX^\mu \to \calX^\lambda \ | \ \mu, \lambda \in \Lambda\}
\]
is a cofiltered system of simplicial finite sets, then
\[
\Pi_1\Big(\, \lim_{\lambda \in \Lambda} \calX^\lambda\Big) = \lim_{\lambda \in \Lambda} \Pi_1\big(\calX^\lambda\big).
\]

\section{Stabilizers}

Any vertex $\sigma$ of $\curv(S)$ may be viewed as a vertex of $\widehat{\curv}(S)$ or of $\checkcurv(S)$, and we let context determine which is meant.

There is ambiguity in the notation $\widehat{\Pure}(\dottedS)_\sigma$, as it could denote the stabilizer of $\sigma$ in the group $\widehat{\Pure}(\dottedS)$, or the profinite completion of $\Pure(\dottedS)_\sigma$.
For arbitrary $\sigma$, it remains unknown if these two groups coincide (but this would follow from the congruence subgroup property).
In the cases we consider here, there is no ambiguity, and we record the following lemma to alleviate any uneasiness.  
See Proposition 6.5 of \cite{Boggi.2006} for a more general statement.
\begin{proposition}\label{ProfiniteStabilizersForPants} Let $\gamma$ be a nonperipheral simple loop in $\dottedS$ that is peripheral in $S$. 
Then 
\[
\left(\widehat{\Pure}(\dottedS)\right)_{\! \gamma}
= \overline{(\Pure(\dottedS)_\gamma)} 
\cong \widehat{\Pure(\dottedS)_\gamma },
\]
and there is a short exact sequence
\[
\begin{tikzpicture}[>=to, line width = .075em, baseline=(current bounding box.center)]
\matrix (m) [matrix of math nodes, column sep=1.5em, row sep = 1em, text height=1.5ex, text depth=0.25ex]
{
1 & \widehat{\ZZ} & \widehat{\Pure}(\dottedS)_\gamma & \widehat{\Pure}(S) & 1. \\
};
\path[->,font=\scriptsize]
(m-1-1) edge 						(m-1-2)
(m-1-2) edge 						(m-1-3)
(m-1-3) edge 						(m-1-4)
(m-1-4) edge						(m-1-5)
	;
\end{tikzpicture}
\]

\end{proposition}
\begin{proof}
We first establish the isomorphism $\smash{\overline{(\Pure(\dottedS)_\gamma)} 
\cong \widehat{\Pure(\dottedS)_\gamma}}$.

Consider the short exact sequence
\begin{equation}\label{StabilizerSequence}
\begin{tikzpicture}[>=to, line width = .075em, baseline=(current bounding box.center)]
\matrix (m) [matrix of math nodes, column sep=1.5em, row sep = 1em, text height=1.5ex, text depth=0.25ex]
{
1  & \ZZ & \Pure(\dottedS)_\gamma  &  \Pure(S)  & 1 \\
};
\path[->,font=\scriptsize]
(m-1-1) edge 						(m-1-2)
(m-1-2) edge node[auto] {$\iota$}		(m-1-3)
(m-1-3) edge node[auto] {$\Phi$}		(m-1-4)
(m-1-4) edge						(m-1-5)
	;
\end{tikzpicture}
\end{equation}
where $\Phi$ is the restriction of the forgetful map $F$ to the stabilizer $\Pure(\dottedS)_\gamma$
and the subgroup $\iota(\ZZ)$ is generated by a Dehn twist in $\gamma$, see \cite{IvanovStability}.
Taking profinite completions is right exact (Proposition 3.2.5 of \cite{RibesZalesskii}), and so we have an exact sequence of profinite completions
\[
\begin{tikzpicture}[>=to, line width = .075em, baseline=(current bounding box.center)]
\matrix (m) [matrix of math nodes, column sep=2em, row sep = 1em, text height=1.5ex, text depth=0.25ex]
{
\widehat{\ZZ} & \widehat{\Pure(\dottedS)_\gamma} & \widehat{\Pure}(S) & 1. \\
};
\path[->,font=\scriptsize]
(m-1-1) edge node[auto] {$\widehat{\iota}$} 	(m-1-2)
(m-1-2) edge node[auto] {$\widehat{\Phi}$}		(m-1-3)
(m-1-3) edge							(m-1-4)
	;
\end{tikzpicture}
\]

Consider the natural surjection 
\[
q \co \widehat{\Pure(\dottedS)_\gamma} \longrightarrow \overline{(\Pure(\dottedS)_\gamma)}. 
\]
By a theorem of Scott, Theorem 3.3 of \cite{Scott,Scott.Correction}, if $G$ is a finitely generated subgroup of $\pi$, then the natural map $\widehat{G} \to \widehat{\pi}$ is injective.
Since $\iota(\ZZ)$ lies in $\calB$, and $\overline{\calB} \cong \widehat{\pi}$ by Theorem \ref{birmanseparable}, we have  a natural isomorphism between $\overline{\iota(\ZZ)} \subset \overline{\calB} \cap \overline{(\Pure(\dottedS)_\gamma)}$ and $\widehat{\ZZ}$.
So $q \circ \widehat{\iota}$, and hence $\widehat{\iota}$, is injective, and we have a short exact sequence
\[
\begin{tikzpicture}[>=to, line width = .075em, baseline=(current bounding box.center)]
\matrix (m) [matrix of math nodes, column sep=2em, row sep = 1em, text height=1.5ex, text depth=0.25ex]
{
1 & \widehat{\ZZ} & \widehat{\Pure(\dottedS)_\gamma} & \widehat{\Pure}(S) & 1. \\
};
\path[->,font=\scriptsize]
(m-1-1) edge							(m-1-2)
(m-1-2) edge node[auto] {$\widehat{\iota}$} 	(m-1-3)
(m-1-3) edge node[auto] {$\widehat{\Phi}$}		(m-1-4)
(m-1-4) edge							(m-1-5)
	;
\end{tikzpicture}
\]

The restriction of $\widehat{\Pure}(\dottedS) \to \widehat{\Pure}(S)$ to $\overline{(\Pure(\dottedS)_\gamma)}$ is a surjection through which the map $\widehat{\Pure(\dottedS)_\gamma} \to \widehat{\Pure(S)}$ factors.
It follows that the kernel of the map $\overline{(\Pure(\dottedS)_\gamma)} \to \widehat{\Pure}(S)$ is precisely $\overline{\iota(\ZZ)} \cong \widehat{\ZZ}$, and the Five Lemma gives us $\overline{(\Pure(\dottedS)_\gamma)} 
\cong \widehat{\Pure(\dottedS)_\gamma}$.

We now turn to the equality 
\[
\left(\widehat{\Pure}(\dottedS)\right)_{\! \gamma}
= \overline{(\Pure(\dottedS)_\gamma)}.
\]
The stabilizer of $\gamma$ in $\widehat{\Pure}(\dottedS)$ is closed.
So it suffices to show that $\Pure(\dottedS)_\gamma$ is dense therein.
Let $\mathfrak{a}$ be an element of $\calA(\dottedS)$ that is normal in $\Pure(\dottedS)$, and consider the finite simplicial set $\curv^\mathfrak{a}(\dottedS) = \curv(\dottedS) / \mathfrak{a}$, on which $\Pure(\dottedS) / \mathfrak{a}$ acts.
The vertices of $\curv^\mathfrak{a}(\dottedS)$ are simply $\mathfrak{a}$--orbits of vertices of $\curv(\dottedS)$.
Let $\varphi\mathfrak{a}$ be an element of $\big( \Pure(\dottedS)/\mathfrak{a} \big)_{\! \mathfrak{a} \gamma}$.
So $\varphi\mathfrak{a} \cdot \mathfrak{a} \gamma = \varphi \mathfrak{a} \gamma = \mathfrak{a} \gamma$.
Now, since $\mathfrak{a}$ is normal, we have $\varphi \mathfrak{a} = \mathfrak{a} \varphi$, and so $\mathfrak{a} \varphi \gamma = \mathfrak{a} \gamma$.
So there is a $\psi$ in $\mathfrak{a}$ such that $\psi \varphi \gamma = \gamma$.
But then $\psi\varphi$ is an element of $\Pure(\dottedS)_\gamma$ that maps to $\varphi \mathfrak{a}$, and 
the map $\Pure(\dottedS)_\gamma \to \big( \Pure(\dottedS)/\mathfrak{a} \big)_{\! \mathfrak{a} \gamma}$ is surjective.
Since $\mathfrak{a}$ was an arbitrary normal element of $\calA$, and the normal subgroups are cofinal in $\calA$, we obtain density of $\Pure(\dottedS)_\gamma$ in $\smash{(\widehat{\Pure}(\dottedS))_\gamma}$.
\end{proof}

The last paragraph of the proof in fact establishes the following general lemma, which we record for posterity.

\begin{lemma}
Let $G$ be a discrete group, $\widetilde G$ a profinite completion of $G$, and let $\mathcal{U}$ be the set of subgroups of $G$ corresponding to the open subgroups of $\widetilde G$.
Suppose that $G$ acts simplicially with finite orbits on a simplicial set $\mathcal{X}_\bullet$ and define 
\[
\widetilde{\mathcal{X}}_\bullet = \varprojlim_{U \in \mathcal{U}} \mathcal{X}_\bullet/U.
\]
(This is a profinite space equipped with a natural $G$--equivariant map $\iota \co \mathcal{X}_\bullet \to \widetilde{\mathcal{X}}_\bullet$ on which $\widetilde G$ acts simplicially and continuously.)
Then, given a simplex $\sigma$ of $\mathcal{X}_\bullet$, the stabilizer $\widetilde{G}_{\iota(\sigma)}$ of $\iota(\sigma)$ is the closure in $\widetilde G$ of the stabilizer $G_\sigma$ of $\sigma$.
\qed
\end{lemma}

\section{Fundamental groups of quotients}
We record here some results that we need in the proof of our main theorem (Theorem \ref{sequester}) in the next section.

Let $S = S_{g,n}$ be a Riemann surface of finite type of genus $g$ with $n$ punctures.  If $g =0$, let $h = n-4$.  If $n=0$, let $h = 2g-2$.  
If $g \geq 1$ and $n \neq 0$,  let $h = 2g-3+n$.  

\begin{theorem}[Harer, Theorem 3.5 of \cite{Harer}]\label{wedgeofspheres} The space $\curv(S)$ is homotopy equivalent to a wedge of spheres of dimension $h$.
In particular, if $\dim \curv(S) = 3g - 3 + n \geq 2$, then $\curv(S)$ is simply--connected.
\qed
\end{theorem}

\begin{theorem}[M. A. Armstrong, Theorem 3 of \cite{Armstrong.1965}]\label{armstrongtheorem} Let $\calX$ be a simply connected simplicial complex.  
Let $G$ be a group of simplicial homeomorphisms of $\calX$ acting without inversions and let $G_*$ be the normal subgroup of $G$ generated by elements with nonempty fixed--point set.  Then $\pi_1(\calX/G) \cong G/G_*$.
\qed
\end{theorem}
\noindent This theorem allows us to view $\pi_1(\curv^\mathfrak{a})$ as a quotient of $\mathfrak{a}$ by the subgroup generated by its reducible elements.

\section{Cornering the congruence kernel}\label{cornering}

Boggi has observed the following theorem.
\begin{theorem}[Theorem 7.2 of \cite{Boggi1111.2372}]\label{BoggiFundamentalGroup} 
If $S = S_{g,n}$ with $\chi(S) < 0$ and $3g-3 + n \geq 2$, then
\[
\Pi_1(\widehat{\curv}(S)) =\lim_\calA \, \widehat{\pi}_1\big(\curv^\mathfrak{a}(S)\big)
= 1.
\]
\end{theorem}
\begin{proof}
By Theorem \ref{armstrongtheorem}, we have a surjection $\widehat{\mathfrak{a}} \twoheadlongrightarrow \widehat{\pi}_1\big(\curv^\mathfrak{a}\big)$. 
Since inverse limit functors are exact on profinite groups (see Proposition 2.2.4 of \cite{RibesZalesskii}), we have
\[
\lim_\calA \, \widehat{\mathfrak{a}}
\twoheadlongrightarrow \lim_\calA \widehat{\pi}_1\big(\curv^\mathfrak{a}\big) 
. 
\]
But $\lim_\calA \, \widehat{\mathfrak{a}} = 1$.
\end{proof}
\noindent

Given $\mathfrak{a}$ in $\calA$, we let $\mathfrak{a}_*$ denote the subgroup of $\mathfrak{a}$ generated by reducible elements.\footnote{An element is \textit{reducible} if it preserves an essential $1$--dimensional submanifold up to isotopy, or, equivalently, it fixes a simplex in the curve complex.}
The rest of this section is devoted to the proof of the following theorem.

\begin{theorem}\label{sequester} Let $S$ be a smooth affine algebraic curve.
Let $\calP$ be the set of nonperipheral simple loops on $\dottedS$ that are peripheral in $S$.
Then
\[
\mathbf{K}(\dottedS) 
\subset \bigcap_{\gamma \, \in \calP} \, \widehat{\Pure}(\dottedS)_\gamma 
\]
and 
$
\Pi_1(\checkcurv(\mathring{S})) 
= \smash{\lim_{\calK} \widehat{\pi}_1(\curv^\kappa(\dottedS)) 
= 1.}
$
\end{theorem}

\begin{lemma}\label{CongruenceKernelAsIntersection} 
\[
\lim_{\calK} \widehat{\kappa} = \mathbf{K}.
\]
\end{lemma}
\begin{proof}
The closure $\overline{\kappa}$ of $\kappa$ in $\smash{\widehat{\Pure}(\dottedS)}$ is isomorphic to $\widehat{\kappa}$, and so $\lim_{\calK} \widehat{\kappa} = \lim_\calK \overline{\kappa}$.
This limit is the intersection of all of the finite index subgroups of $\widehat{\Pure}(\dottedS)$ obtained by pulling back finite index subgroups of $\checkGamma(\dottedS)$, which is precisely the congruence kernel $\mathbf{K}$.
\end{proof}

\begin{lemma}\label{ClosureOfProductsIntersectTrivially} 
For any $\mathfrak{h}$ in $\calK$, we have $\lim_{\kappa \in \calK} \overline{\calB^\kappa} \,  \overline{\mathfrak{h}_\gamma} = \overline{\mathfrak{h}_\gamma}$.
\end{lemma}
\begin{proof} 
By Theorem \ref{birmanseparable},
\[\lim_{\kappa \in \calK} \overline{\calB^\kappa} \,  \overline{\mathfrak{h}_\gamma}
= \lim_{\mathfrak{a} \in \calA} \overline{\calB^\mathfrak{a}} \,  \overline{\mathfrak{h}_\gamma}, 
\]
and clearly
\[
\lim_{\mathfrak{a} \in \calA} \overline{\calB^\mathfrak{a}} \,  \overline{\mathfrak{h}_\gamma}
= \bigcap_{\mathfrak{a} \in \calA} \overline{\calB^\mathfrak{a}} \,  \overline{\mathfrak{h}_\gamma} 
\supset \overline{\mathfrak{h}_\gamma}.
\]

On the other hand, if $\mathfrak{H}$ is a finite index subgroup of $\smash{\widehat\Pure(\dottedS)}$ that contains $\smash{\overline{\mathfrak{h}_\gamma}}$, then there is a finite index subgroup of $\mathfrak{H}$ of the form $\overline{\calB^\mathfrak{a}}  \, \overline{\mathfrak{h}_\gamma}$, obtained by letting $\mathfrak{a} = \mathfrak{H} \cap \Pure(\dottedS)$.
But since $\overline{\mathfrak{h}_\gamma}$ is closed, it is the intersection of all of the finite index subgroups containing it, and we conclude that 
\[
\overline{\mathfrak{h}_\gamma}= \bigcap_{\mathfrak{a} \in \calA} \overline{\calB^\mathfrak{a}} \,  \overline{\mathfrak{h}_\gamma}. \qedhere
\] 
\end{proof}

\begin{proof}[Proof of Theorem \ref{sequester}] Consider the short exact sequence 
\[
\begin{tikzpicture}[>=to, line width = .075em, baseline=(current bounding box.center)]
\matrix (m) [matrix of math nodes, column sep=1.5em, row sep = 1em, text height=1.5ex, text depth=0.25ex]
{
1 & \kappa_* & \kappa & \pi_1(\curv^\kappa) & 1 \\
};
\path[->,font=\scriptsize]
(m-1-1) edge 	(m-1-2)
(m-1-2) edge 	(m-1-3)
(m-1-3) edge	(m-1-4)
(m-1-4) edge	(m-1-5)
	;
\end{tikzpicture}
\]
given by Theorem \ref{armstrongtheorem}.
Passing to profinite completions is right exact (see Proposition 3.2.5 of \cite{RibesZalesskii}), and so we have an exact sequence
\[
\begin{tikzpicture}[>=to, line width = .075em, baseline=(current bounding box.center)]
\matrix (m) [matrix of math nodes, column sep=1.5em, row sep = 1em, text height=1.5ex, text depth=0.25ex]
{
\widehat{\kappa_*} & \widehat{\kappa} & \widehat{\pi}_1(\curv^\kappa) & 1. \\
};
\path[->,font=\scriptsize]
(m-1-1) edge 	(m-1-2)
(m-1-2) edge 	(m-1-3)
(m-1-3) edge	(m-1-4)
	;
\end{tikzpicture}
\]
Inverse limits are exact on profinite groups (Proposition 2.2.4 of \cite{RibesZalesskii}), and so, by Lemma \ref{CongruenceKernelAsIntersection}, we have an exact sequence
\begin{equation}\label{profiniteArmstrong}
\begin{tikzpicture}[>=to, line width = .075em, baseline=(current bounding box.center)]
\matrix (m) [matrix of math nodes, column sep=1.5em, row sep = 1em, text height=1.5ex, text depth=0.25ex]
{
\displaystyle{\lim_{\calK}} \, \widehat{\, \kappa_* \, } & \mathbf{K} & 
\displaystyle{\lim_{\calK} \, \widehat{\pi}_1(\curv^\kappa(\dottedS))}
& 1 . \\
};
\path[->,font=\scriptsize]
(m-1-1) edge 	(m-1-2)
(m-1-2) edge 	(m-1-3)
(m-1-3) edge	(m-1-4)
	;
\end{tikzpicture}
\end{equation}

We let $\mathfrak{R}_*$ and $\mathfrak{R}_\gamma \subset \mathfrak{R}_*$ be the images in $\mathbf{K}$ of  $\lim_{\calK} \widehat{\, \kappa_* \, }$ and $\lim_{\calK} \widehat{\kappa}_\gamma$, respectively.

By Corollary \ref{CofinalCorollary},  the system of subgroups 
\[
\{ \calB^\kappa \kappa_\gamma \ | \  \kappa \mbox{\ in } \calK \mbox{\ and } \calB^\kappa = \kappa \cap \calB \}
\]
is cofinal in $\calK$.
Now, $\overline{\calB^\kappa} \cong \widehat{\calB^\kappa}$, and making use of the forgetful map $\Pure(\dottedS) \to \Pure(S)$, it is easily seen that $\overline{\kappa_\gamma} \cong \widehat{\kappa_\gamma}$.
We also have $\widehat{\calB^\kappa \kappa_\gamma} \cong  \overline{ \calB^\kappa \kappa_\gamma} \subset \widehat{\Pure}$, and since $\overline{\calB^\kappa}\overline{\kappa_\gamma}$ is dense in $\overline{ \calB^\kappa \kappa_\gamma}$, the two are equal, as both are closed.
The congruence kernel $\mathbf{K}$ is then
\[
\mathbf{K} 
= \lim_{\kappa \in \calK} \widehat{\calB^\kappa \kappa_\gamma} 
= \lim_{\kappa \in \calK} \overline{\calB^\kappa} \overline{\kappa_\gamma}.
\]
Now, for a fixed $\mathfrak{h}$ in $\calK$, the subset $\{ \mathfrak{h} \cap \kappa \ | \ \kappa \in \calK \}$ is cofinal in $\calK$, by Lemma \ref{finiteintersections}.
It follows that, for any such $\mathfrak{h}$,
\begin{equation}\label{PartialLimit}
\lim_{\kappa \in \calK} \overline{\calB^\kappa} \overline{\kappa_\gamma}  
\subset \lim_{\kappa \in \calK} \overline{\calB^\kappa} \,  \overline{\mathfrak{h}_\gamma}
\end{equation}
and we have
\begin{align*}
\mathbf{K} &
=\lim_{\kappa \in \calK} \overline{\calB^\kappa} \overline{\kappa_\gamma}   \\
& \subset \lim_{\mathfrak{h} \in \calK} \, \lim_{\kappa \in \calK} \overline{\calB^\kappa} \,  \overline{\mathfrak{h}_\gamma}  &   \mbox{by \eqref{PartialLimit}} \\
& \subset \lim_{\mathfrak{h} \in \calK} \, \widehat{\mathfrak{h}}_\gamma & \mbox{by Lemma \ref{ClosureOfProductsIntersectTrivially}}\\
& \subset  \mathfrak{R}_\gamma  \\
& \subset  \mathfrak{R}_* 
\end{align*}

Together with \eqref{profiniteArmstrong}, we have the short exact sequence
\[
\begin{tikzpicture}[>=to, line width = .075em, baseline=(current bounding box.center)]
\matrix (m) [matrix of math nodes, column sep=1.5em, row sep = 1em, text height=1.5ex, text depth=0.25ex]
{
\mathbf{K} & \mathbf{K} & 
\displaystyle{\lim_{\calK} \, \widehat{\pi}_1(\curv^\kappa(\dottedS))} 
& 1 . \\
};
\path[->,font=\scriptsize]
(m-1-1) edge 	(m-1-2)
(m-1-2) edge 	(m-1-3)
(m-1-3) edge	(m-1-4)
	;
\end{tikzpicture}\qedhere
\]
\end{proof}

\begin{corollary}
If $S$ is affine, then, for any prime number $p$,
\[
\rmH_1(\checkcurv(\dottedS) ; \mathbb{F}_p) 
:= \lim_\calK \, \rmH_1 \big(\curv^\kappa(\dottedS) ; \mathbb{F}_p \big)
= 0.  
\]
\end{corollary}
\begin{proof}
The surjections $\smash{\widehat{\pi}_1(\curv^\kappa(\dottedS)) \twoheadlongrightarrow \rmH_1 \big(\curv^\kappa(\dottedS) ; \mathbb{F}_p \big)}$, exactness of inverse limits on profinite groups, and Theorem \ref{sequester} prove the corollary.
\end{proof}

Let $S$ be an affine curve and let $\Sigma$ be its projective completion.
The complement $\Sigma - S$ is a finite set of $n$ points, for some $n$.
There is a map $\Pure(S) \to \Pure(\Sigma)$ obtained by extending homeomorphisms from $S$ to $\Sigma$.
The kernel $\mathfrak{b}(S)$ of this map is the \textit{braid group of $\Sigma$ on $n$ strands}. 
The elements of $\mathfrak{b}(S)$ are called \textit{braids}.

If $\gamma$ is a nonperipheral simple closed loop in $\dottedS$ that is peripheral in $S$, then a Dehn twist $T_\gamma$ in $\gamma$ lies in $\mathfrak{b}(\dottedS)$.
In fact, the Dehn twists in such curves generate $\mathfrak{b}(\dottedS)$.
Since $\widehat{\Pure}(\dottedS)_\gamma$ lies in the centralizer of $T_\gamma$,
we have the following corollary of Theorem \ref{sequester}.
\begin{corollary}\label{Corollary:BraidCentralizer} When $S$ is affine, the congruence kernel $\mathbf{K}(\dottedS)$ centralizes $\mathfrak{b}(\dottedS)$. \qed
\end{corollary}

\section{Kernels depend only on the genus of affine curves}\label{IndependentKernels}
If $S$ is affine, it follows from Matsumoto's exact sequence
\[
\begin{tikzpicture}[>=to, line width = .075em, baseline=(current bounding box.center)]
\matrix (m) [matrix of math nodes, column sep=1.5em, row sep = 1em, text height=1.5ex, text depth=0.25ex]
{
1 & \widehat{\pi} & \checkGamma(\dottedS) & \checkGamma(S) & 1 \\
};
\path[->,font=\scriptsize]
(m-1-1) edge 	(m-1-2)
(m-1-2) edge 	(m-1-3)
(m-1-3) edge	(m-1-4)
(m-1-4) edge	(m-1-5)
	;
\end{tikzpicture}
\]
that $\smash{\mathbf{K}(\dottedS)}$ and $\mathbf{K}(S)$ are isomorphic, and so the congruence subgroup problem depends only on the genus in the affine case.
It is a deeper theorem of Boggi \cite{BoggiClosed} and, independently, Hoshi and Mochizuki \cite{HoshiMochizuki2009}, that this is true in all cases.

We provide a new proof of this fact in the affine case. 
We are grateful to Marco Boggi for suggesting that our techniques should accomplish this.

\begin{theorem}[Matsumoto]\label{KernelIndependence} If $S$ is affine, then $\mathbf{K}(\dottedS) \cong \mathbf{K}(S)$ and there is a  short exact sequence
\begin{equation}\label{MatsumotoSequence}
\begin{tikzpicture}[>=to, line width = .075em, baseline=(current bounding box.center)]
\matrix (m) [matrix of math nodes, column sep=1.5em, row sep = 1em, text height=1.5ex, text depth=0.25ex]
{
1 & \widehat{\pi} & \checkGamma(\dottedS) & \checkGamma(S) & 1. \\
};
\path[->,font=\scriptsize]
(m-1-1) edge 	(m-1-2)
(m-1-2) edge 	(m-1-3)
(m-1-3) edge	(m-1-4)
(m-1-4) edge	(m-1-5)
	;
\end{tikzpicture}
\end{equation}
\end{theorem}

\begin{proof}[Proof of Theorem \ref{KernelIndependence}]
By Theorem \ref{sequester}, we have $\mathbf{K}(\dottedS) \subset \widehat{\Pure}(\dottedS)_\gamma$, where $\gamma$ bounds a pair of pants in $\dottedS$.

By Proposition \ref{ProfiniteStabilizersForPants}, we have the short exact sequence
\begin{equation}
\begin{tikzpicture}[>=to, line width = .075em, baseline=(current bounding box.center)]
\matrix (m) [matrix of math nodes, column sep=1.5em, row sep = 1em, text height=1.5ex, text depth=0.25ex]
{
1 & \widehat{\ZZ} & \widehat{\Pure}(\dottedS)_\gamma & \widehat{\Pure}(S) & 1 \\
};
\path[->,font=\scriptsize]
(m-1-1) edge 							(m-1-2)
(m-1-2) edge 							(m-1-3)
(m-1-3) edge node[auto] {$\widehat{\Phi}$}		(m-1-4)
(m-1-4) edge							(m-1-5)
	;
\end{tikzpicture}
\end{equation}
where $\smash{\widehat{\Phi}}$ is the restriction of the forgetful map $\smash{\widehat{F}}$ to $\smash{\widehat{\Pure}(\dottedS)_\gamma}$.
Exactness on the left follows from Theorem \ref{birmanseparable} and Scott's theorem that $\pi$ is subgroup separable (see the proof of Proposition \ref{ProfiniteStabilizersForPants}).
In fact, these two theorems imply that the $\widehat{\ZZ}$ in this exact sequence injects into $\checkGamma(\dottedS)$.
In particular, $\mathbf{K}(\dottedS)$ intersects this $\widehat{\ZZ}$ trivially.
So $\mathbf{K}(\dottedS)$ injects into $\widehat{\Pure}(S)$.

Given a congruence subgroup $\eta \subset \Pure(S)$, the subgroup $\widehat{F}^{-1}(\eta)$ is a congruence subgroup of $\Pure(\dottedS)$, by Lemma \ref{pullcongruence}.
So,
\begin{align*}
\mathbf{K}(\dottedS) 
& = \bigcap_{\calK(\dottedS)} \overline{\kappa} \\
& \subset \bigcap_{\calK(S)} \widehat{F}^{-1}(\overline{\eta}) \\
& = \widehat{F}^{-1}\bigcap_{\calK(S)} \overline{\eta} \\
& = \widehat{F}^{-1}(\mathbf{K}(S))
\end{align*}
and so $\widehat{F}(\mathbf{K}(\dottedS)) \subset \mathbf{K}(S)$.

On the other hand, by Scholium \ref{pushcongruence}, for every $\kappa$ in $\calK(\dottedS)$, the subgroup $\widehat{\Phi}(\kappa_\gamma)$ is a congruence subgroup of $\Pure(S)$.
So 
\[
\mathbf{K}(S) 
= \bigcap_{\calK(S)} \overline{\eta} 
\subset \bigcap_{\calK(\dottedS)} \widehat{\Phi}(\overline{\kappa_\gamma}) ,
\]
which yields
\begin{align*}
\widehat{\Phi}^{-1}(\mathbf{K}(S))
& \subset \widehat{\Phi}^{-1} \Bigg(\bigcap_{\calK(\dottedS)} \widehat{\Phi}(\overline{\kappa_\gamma}) \Bigg)\\
& = \bigcap_{\calK(\dottedS)} \widehat{\Phi}^{-1}\left(\widehat{\Phi}(\overline{\kappa_\gamma})\right) \\
& = \bigcap_{\calK(\dottedS)} \widehat{\ZZ}  \, \overline{\kappa_\gamma} 
.
\end{align*}

Now, let $x$ be an element of $\bigcap_{\calK(\dottedS)} \widehat{\ZZ}  \, \overline{\kappa_\gamma}$.  
So for each $\kappa$ in $\calK(\dottedS)$, there are $z_\kappa$ in $\widehat{\ZZ}$ and $x_\kappa$ in $\overline \kappa_\gamma$ such that $x = z_\kappa x_\kappa$.
By Lemma \ref{finiteintersections}, we may enumerate the elements of $\calK(\dottedS)$ and take intersections to obtain a \textit{nested} sequence $\{\kappa^n\}$ in $\calK(\dottedS)$ such that 
\[
\bigcap_{n=1}^\smallinfinity  \overline{\kappa^n_\gamma}
= \bigcap_{\calK(\dottedS)}  \overline{\kappa_\gamma}
= \mathbf{K}(\dottedS) 
.
\]
Let $z_n = z_{\kappa^n}$ and $x_n = x_{\kappa^n}$.
After passing to a subsequence, the $z_n$ converge to some $z$ in $\widehat{\ZZ}$, since $\widehat{\ZZ}$ is compact.
Since the $\overline{\kappa^n_\gamma}$ are compact and nested, we may pass to a further subsequence so that the $x_n$ converge to some $x_\smallinfinity$ in $\smash{\mathbf{K}(\dottedS)}$.
We conclude that $x = z x_\smallinfinity$, which transparently lies in $\widehat{\ZZ} \mathbf{K}(\dottedS)$.
So 
\[
\bigcap_{\calK(\dottedS)} \widehat{\ZZ}  \, \overline{\kappa_\gamma} 
=  \widehat{\ZZ} \, \cdot  \! \!  \bigcap_{\calK(\dottedS)}  \overline{\kappa_\gamma} 
= \widehat{\ZZ} \mathbf{K}(\dottedS).
\]

All together, we have
\[
\smash{
\widehat{\Phi}^{-1}(\mathbf{K}(S))
\subset \bigcap_{\calK(\dottedS)} \widehat{\ZZ}  \, \overline{\kappa_\gamma} 
= \widehat{\ZZ} \, \cdot  \! \!  \bigcap_{\calK(\dottedS)}  \overline{\kappa_\gamma}
= \widehat{\ZZ} \, \mathbf{K}(\dottedS)
}
\]
and so 
\[
\mathbf{K}(S) \subset \widehat{F}(\widehat{\ZZ} \, \mathbf{K}(\dottedS)) = \widehat{F}(\mathbf{K}(\dottedS)).
\]

\noindent
We conclude that $\widehat{F}\big|_{\mathbf{K}(\dottedS)} \co \mathbf{K}(\dottedS) \to \mathbf{K}(S)$ is an isomorphism.

\bigskip
\noindent 
To see that we recover the exact sequence, consider the diagram
\[
\begin{tikzpicture}[>=to, line width = .075em, baseline=(current bounding box.center), descr/.style={fill=white,inner sep=2.5pt}]
\matrix (m)  [matrix of math nodes, row sep=2.5em,
column sep=2.5em, text height=1.5ex, text depth=0.25ex]
{
1  & \widehat{\dottedpi}  & \widetilde \Pure(\doubledottedS) & \checkGamma(\dottedS)  & \widehat{\Pure}(\dottedS)\\
1 & \widehat{\pi} & \widetilde \Pure(\dottedS) & \checkGamma(S) & \widehat{\Pure}(S) \\
};
\path[->, font=\scriptsize]
(m-1-1)	edge							(m-1-2)
(m-1-2)	edge							(m-1-3)
		edge							(m-2-2)
(m-1-3)	edge[->>]						(m-1-4)
		edge	node[left] {$\widetilde{F}$}		(m-2-3)
(m-1-4)	edge	node[auto] {$\check{F}$}		(m-2-4)
		edge[->>] node[descr] {$\xi$}		(m-2-3)
(m-1-5)	edge	node[auto] {$\widehat{F}$}	(m-2-5)	
		edge	node[auto] {$\mathring \rho$}		(m-1-4)
(m-2-1)	edge							(m-2-2)
(m-2-2)	edge							(m-2-3)
(m-2-3)	edge[->>] node[below] {$q$}		(m-2-4)
(m-2-5)	edge	node[auto] {$\rho$}			(m-2-4)
	;
\end{tikzpicture}
\]
where the diagonal map $\xi$ is the map from the proof of Theorem \ref{birmanseparable} and $\check{F} = q \circ \xi$.
Other than the right--most square, the diagram commutes by definition.
To see that this square commutes, note that 
\[
\check{F} \circ \mathring \rho \big|_{\Pure(\dottedS)} 
= F 
= \rho \circ \widehat{F} \big|_{\Pure(\dottedS)} \, ,
\]
and that, by the universal property of the profinite completion, there is a unique continuous homomorphism to $\checkGamma(S)$ extending $F$, which must be 
$\check{F} \circ \mathring \rho \big|_{\Pure(\dottedS)} 
= \rho \circ \widehat{F} \big|_{\Pure(\dottedS)}$.

Suppose that $\varphi$ is an element of $\ker(\check{F}) - \calB^\star$. 
Pick an element $\varphi'$ in $\widehat{\Pure}(\dottedS)$ in the preimage of $\varphi$.
Since the diagram commutes, we know that $\widehat{F}(\varphi')$ lies in $\mathbf{K}(S)$.
Since 
\[
\widehat{F}\big|_{\mathbf{K}(\dottedS)} \co \mathbf{K}(\dottedS) \to \mathbf{K}(S)
\]
is an isomorphism, we may choose an element $\psi$ of $\smash{\overline{\calB} = \ker(\widehat{F})}$ such that $\psi\varphi'$ lies in $\mathbf{K}(\dottedS)$.
But this means that we may multiply $\varphi$ by an element of $\calB^\star$ to obtain the trivial element, contradicting our choice of $\varphi$.
We conclude that $\ker(\check{F}) = \calB^\star \cong \widehat{\pi}$.
\end{proof}

\section{The congruence subgroup problem in genus one}\label{Section:GenusOne}

In this section, marked points are more convenient than punctures.
We let $S_{g,n}$ be the oriented surface of genus $g$ with a set $X_{g,n}$ of $n$ marked points, and let $\pi_{g,n} = \pi_1(S_{g,n} - X_{g,n})$.
From this viewpoint, $\Mod(S_{g,n})$ is the group of orientation preserving homeomorphisms of the pair $(S_{g,n}, X_{g,n})$ up to isotopy of pairs, see Chapter 2 of \cite{FarbMargalit}.
We let $\Pure_{g,n} = \Pure(S_{g,n})$ be the kernel of the action of $\Mod(S_{g,n})$ on $\mathrm{H}_1(S_{g,n} - X_{g,n}; \mathbb{F}_3)$.
If $\sigma$ is an essential simple closed curve in $S_{g,n}$, we let $T_\sigma$ be the positive Dehn twist in $\sigma$.

We give a new proof of Asada's theorem \cite{Asada.2001} that the congruence kernel vanishes in genus one.
Our proof uses properties of centralizers in free profinite groups due to Herfort and Ribes \cite{HerfortRibes} (also featured in the proof given in \cite{BuxErshovRapinchuk}) and our theorem on congruence kernels centralizing braids.
It was inspired by the proofs of Lemmata 3.8 and 3.9 of \cite{BoggiLochak}.
A novel feature of the argument is that it deduces the general case from the case of $\Mod(S_{1,2})$ using Matsumoto's exact sequence \eqref{MatsumotoSequence}, rather than beginning with $\Mod(S_{1,1})$.
A similar strategy of increasing the number of punctures is taken in Boggi's proof of the genus--$2$ case \cite{BoggiOpen,BoggiClosed}, and this leads us to speculate that the general case may yield to imposing greater and greater restrictions on the congruence kernel in more and more highly punctured surfaces. 
Boggi's argument in \cite{BoggiOpen} also applies to the genus--one case (see remark 3.4 of \cite{BoggiOpen} and the following observation there).

\begin{theorem}[Asada \cite{Asada.2001}] 
$\Mod(S_{1,n})$ has the congruence subgroup property if $n \geq 1$.
\end{theorem}
\begin{proof}
The quotient
$
S_{1,2} \to S_{0,5}
$
by the hyperelliptic involution induces an injection
$
\Pure_{0,5} \to \Mod(S_{1,2})
$
onto a finite index subgroup $\kappa \subset \Mod(S_{1,2})$ (see \cite{Birman.Hilden.1973} and \cite{MaclachlanHarvey} for more general discussions of maps induced by branched covers).
The subgroup $\kappa$ contains the kernel of the action of $\Mod(S_{1,2})$ on $\mathrm{H}_1(S_{1,2} - X_{1,2}; \mathbb{F}_2)$ and is hence a congruence subgroup.
We let
$
\zeta \co \widehat{\kappa} \to \smash{\widehat{\Pure}_{0,5}}
$
be the induced isomorphism.
Since $\kappa$ is congruence, $\widehat{\kappa}$ contains $\mathbf{K}_{1,2} = \mathbf{K}(S_{1,2})$.

\begin{figure}
\begin{center}
\input{Hyperelliptic.pdftex_t}
\end{center}
\end{figure}

Let $\alpha$ and $\beta$ be the essential simple closed curves on $S_{1,2}$ in the figure.
Each bounds a disk containing both marked points.
By Corollary \ref{Corollary:BraidCentralizer}, the congruence kernel $\mathbf{K}_{1,2}$ lies in the $\widehat{\kappa}$--centralizers of $T_\alpha$ and $T_\beta$.
The curves $\alpha$ and $\beta$ descend to curves $\alpha'$ and $\beta'$ on $S_{0,5}$ as shown. 
We have 
	$\zeta(T_\alpha) = T_{\alpha'}^2$ and
	$\zeta(T_\beta) = T_{\beta'}^2$,
and so $\zeta(\mathbf{K}_{1,2})$ lies in the $\smash{\widehat{\Pure}_{0,5}}$--centralizers of $\smash{T_{\alpha'}^2}$ and $\smash{T_{\beta'}^2}$.

Let $z$ be the marked point in $S_{0,5}$ as pictured.  
Forgetting $z$ gives us an exact sequence
\[
\begin{tikzpicture}[>=to, line width = .075em, baseline=(current bounding box.center)]
\matrix (m) [matrix of math nodes, column sep=1.5em, row sep = 1em, text height=1.5ex, text depth=0.25ex]
{
1 & \widehat{\pi}_{\,0,4} & \widehat{\Pure}_{0,5} & \widehat{\Pure}_{0,4} & 1.  \\
};
\path[->,font=\scriptsize]
(m-1-1) edge 				(m-1-2)
(m-1-2) edge 				(m-1-3)
(m-1-3) edge 	node[auto]{$f_{0,5}$}		(m-1-4)
(m-1-4) edge				(m-1-5)
	;
\end{tikzpicture}
\]
The curves $\alpha'$ and $\beta'$ descend to curves $\alpha''$ and $\beta''$ on $S_{0,4}$ with $f_{0,5}(T_{\alpha'}) = T_{\alpha''}$ and $f_{0,5}(T_{\beta'}) = T_{\beta''}$.
So $f_{0,5}\big(\zeta(\mathbf{K}_{1,2})\big)$ centralizes $T_{\alpha''}^2$ and $T_{\beta''}^2$.
Now, $\widehat{\Pure}_{0,4}$ is the free profinite group generated by $T_{\alpha''}$ and $T_{\beta''}$.
By the main theorem of \cite{HerfortRibes}, the centralizer of any power $x^m$ of a basis element $x$ in a free profinite group is the closed subgroup generated by $x$, see also Lemma 2.1.2 of \cite{Nakamura}.  
It follows that the centralizers of the squares of two elements of a basis intersect trivially, and so $\zeta(\mathbf{K}_{1,2})$ lies in $\widehat{\pi}_{\,0,4}$.

We now pick a curve $\gamma$ in $S_{1,2}$ descending to a curve $\gamma\,'$  in $S_{0,5}$ as in the figure.
Again, the congruence kernel $\mathbf{K}_{1,2}$ centralizes $T_\gamma$ and $\zeta(\mathbf{K}_{1,2})$ centralizes $\smash{T_{\gamma\,'}^2}$.
The Dehn twist $T_{\gamma\,'}$ is a basis element in the free profinite group $\widehat{\pi}_{\,0,4}$, and so the centralizer of $\smash{T_{\gamma\,'}^2}$ is the subgroup $\smash{\overline{\langle T_{\gamma\,'}\rangle}}$, again by \cite{HerfortRibes}.
Since $\zeta\big(\mathbf{K}_{1,2})$ lies in $\widehat{\pi}_{\,0,4}$ and centralizes $\smash{T_{\gamma\,'}^2}$, we conclude that $\zeta(\mathbf{K}_{1,2})$ lies in $\smash{\overline{\langle T_{\gamma\,'}\rangle}}$.
Now, we have $\zeta(\widehat{\kappa}) \cap \overline{\langle T_{\gamma\,'}\rangle} = \overline{\langle T_{\gamma\,'}^2\rangle}$, and so $\zeta(\mathbf{K}_{1,2}) \subset \overline{\langle T_{\gamma\,'}^2\rangle}$.
Since $\overline{\langle T_{\gamma\,'}^2\rangle} = \zeta\big(\overline{\langle T_{\gamma}\rangle}\big)$,  we conclude that $\mathbf{K}_{1,2} \subset \overline{\langle T_{\gamma}\rangle}$.
But $\overline{\langle T_{\gamma}\rangle}$ lies in the kernel of 
\[
\begin{tikzpicture}[>=to, line width = .075em, baseline=(current bounding box.center)]
\matrix (m) [matrix of math nodes, column sep=1.5em, row sep = 1em, text height=1.5ex, text depth=0.25ex]
{
1 & \widehat{\pi}_{1,1} & \widehat{\Pure}_{1,2} & \widehat{\Pure}_{1,1} & 1,  \\
};
\path[->,font=\scriptsize]
(m-1-1) edge 				(m-1-2)
(m-1-2) edge 				(m-1-3)
(m-1-3) edge 	node[auto]{$f_{1,2}$}		(m-1-4)
(m-1-4) edge				(m-1-5)
	;
\end{tikzpicture}
\]
and, by Theorem \ref{birmanseparable}, the congruence kernel $\mathbf{K}_{1,2}$ intersects $\widehat{\pi}_{1,1}$ trivially.
We conclude that $\mathbf{K}_{1,2} = 1$.

By Theorem \ref{KernelIndependence}, we have $\mathbf{K}(S_{1,n})=1$ for all $n \geq 1$.
\end{proof}

\section{Thoughts on the fundamental group at each level}\label{Section:Thoughts}

Boggi has proven the beautiful theorem that if $\mathfrak{a}$ is a finite index subgroup of $\Mod(S)$, then the rational homology $\mathrm{H}_k(\curv^\mathfrak{a}\, ; \QQ)$ vanishes in the range $0 \leq k \leq -\chi(S)$ if $S$ is closed and $0 \leq k \leq -\chi(S) -1$ otherwise, see Lemma 5.5 of \cite{Boggi.2006}.\footnote{The proof of the main theorem of \cite{Boggi.2006} contains a gap, but the proof of Lemma 5.5 of \cite{Boggi.2006} is complete. See \cite{Abramovich.MR2242960}.}
In light of this and the fact that $\widehat{\curv}(S)$ is simply--connected, it is reasonable to ask if the $\curv^\mathfrak{a}$ are simply--connected.
The proofs that $\widehat{\curv}(S)$ and $\checkcurv(\doubledottedS)$ are simply--connected leave open the possibility that $\pi_1(\curv^\mathfrak{a})$ is nontrivial (perhaps infinite) and so new ideas are needed to answer this question.
We record some of our thoughts here.

The simplicial complex $\curv(S)$ comes equipped with the weak topology.
There is another natural topology; the topology induced by declaring each simplex to be a regular Euclidean simplex with edges of length one, and taking the induced path metric.

Let $S$ be such that $\dim\curv(S) \geq 2$.

Let $\mathfrak{a}$ be in $\calA$ and $\sigma$ a nonperipheral simple loop in $S$.
The group $\mathfrak{a}_\sigma$ has finite index in $\Pure(S)_\sigma$, which acts cocompactly on the star of $\sigma$, and so the simplicial complex $\curv^{\mathfrak{a}_*} = \curv(S)/\mathfrak{a}_*$ is locally finite.
In particular, the complex $\curv^{\mathfrak{a}_*}$ is compact if and only if it has finite diameter in the metric induced by that of $\curv(S)$.
By Armstrong's theorem, $\pi_1(\curv^\mathfrak{a})$ is finite if and only if $\curv^{\mathfrak{a}_*}$ is compact.

There is a Serre fibration $\curv(\dottedS) \to \curv(S)$ whenever $S$ is projective \cite{KentLeiningerSchleimer}.
When $S$ is affine, there is no map from $\curv(\dottedS)$ to $\curv(S)$, but there is a $1$--dense subcomplex $\durv(\dottedS)$ for which there is a map
\[
p \co\durv(\dottedS) \to \curv(S),
\]
see \cite{KentLeiningerSchleimer}.
Each fiber of this map is $1$--dense in $\curv(\dottedS)$, and
so $\pi_1(\curv^\mathfrak{a})$ will be finite if and only if a fiber $\mathcal{T}$ of $p$ projects to a set of finite diameter in $\curv^{\mathfrak{a}_*}$.

{\small
\bibliographystyle{plain}
\bibliography{procongruence}
}

\bigskip

\noindent Department of Mathematics, University of Wisconsin, Madison, WI 53706 
\newline \noindent  \texttt{rkent@math.wisc.edu}

\end{document}